\documentclass[12pt]{amsproc}

\usepackage{color}

\newtheorem{theorem}{\sc Theorem}[section]
\newtheorem{lemma}[theorem]{\sc Lemma}
\newtheorem{proposition}[theorem]{\sc Proposition}
\newtheorem{corollary}[theorem]{\sc Corollary}
\newtheorem{remark}[theorem]{\sc Remark}
\newtheorem{hypothesis}[theorem]{\sc Hypothesis}

\newtheorem*{thmOp}{\sc Theorem \ref{Op}}
\newtheorem*{thmB}{\sc Theorem \ref{p_fitt}}

\newcommand{\ep}{\epsilon}
\newcommand{\PSL}{\mathrm{PSL}}

\newcommand{\al}{\alpha}
\newcommand{\pr }{\mathrm{Pr} }
\newcommand{\Gal }{C_G(\al) }
\newcommand{\Qal }{C_Q(\al) }
\newcommand{\F}{\mathbb{F}}

\begin{document}

\title[Commuting probability]{Finite groups, commuting probability, and coprime automorphisms}
\thanks{The first and the third authors are members of GNSAGA (INDAM) and were funded by  Project 2022PSTWLB (subject area: PE - Physical Sciences and
Engineering) ``Group Theory and Applications". The second author was partially supported by NSF grant DMS 1901595 and Simons Foundation Fellowship 609771.
The fourth author was  supported by  FAPDF and CNPq.
}

\author[E. Detomi]{Eloisa Detomi}
\address{Dipartimento di Matematica \lq\lq Tullio Levi-Civita\rq\rq, Universit\`a degli Studi di Padova, Via Trieste 63, 35121 Padova, Italy} 
\email{eloisa.detomi@unipd.it}
\author[R. M. Guralnick]{Robert M. Guralnick}
\address{ Department of Mathematics, University of
Southern California, Los Angeles, CA 90089-2532, USA}
\email{guralnic@usc.edu}

\author[M. Morigi]{Marta Morigi}
\address{Dipartimento di Matematica, Universit\`a di Bologna\\
Piazza di Porta San Donato 5 \\ 40126 Bologna \\ Italy}
\email{marta.morigi@unibo.it}
\author[P. Shumyatsky]{Pavel Shumyatsky}
\address{Department of Mathematics, University of Brasilia\\
Brasilia-DF \\ 70910-900 Brazil}
\email{pavel@unb.br}


\begin{abstract} 
Given two subgroups $H,K$ of a finite group $G$, the probability that a pair of random elements from $H$ and $K$ commutes is denoted by $\pr(H,K)$. Suppose that a finite group $G$ admits a group of coprime automorphisms $A$ and let $\epsilon>0$. We show that,  if for any distinct primes  $p,q\in\pi(G)$ there is an $A$-invariant Sylow $p$-subgroup $P$ and an $A$-invariant Sylow $q$-subgroup $Q$ of $G$ for which $\pr([P,A],[Q,A])\ge\epsilon$, then $F_2([G,A])$ has $\epsilon$-bounded index in $[G,A]$ (Theorem \ref{mainA}). Here $F_2(K)$ stands for the second term of the upper Fitting seris of a group $K$. We also show that,  if $G=[G,A]$ and for any prime $p$ dividing the order of $G$ there is an $A$-invariant Sylow $p$-subgroup $P$ such that $\Pr([P,A], [P,A]^x)\geq\epsilon$ for all $x\in G$, then $G$ is bounded-by-abelian-by-bounded (Theorem \ref{p_fitt}).

 \end{abstract}
 \maketitle

\section{Introduction}

Given two subsets $X,Y$ of a finite group $G$, we write $\pr(X,Y)$ for the
  probability that two random elements $x\in X$ and $y\in Y$ commute. 
The number $\pr(G,G)$ is called the commuting probability of $G$.  It is well-known that $\pr(G,G)\leq5/8$ for any nonabelian group $G$. Another important result is the theorem of P. M. Neumann \cite{neumann} which states that if $G$ is a finite group and $\epsilon$ is  a positive number such that $\pr(G,G)\geq\epsilon$,  
 then $G$ has a normal subgroup $R$ such that both the index $|G:R|$ and the order of the commutator subgroup $[R,R]$ are $\epsilon$-bounded  (see also \cite{eberhard}). When a group $G$ has a structure as in P. M. Neumann's theorem, we say that $G$ is bounded-by-abelian-by-bounded. More generally, throughout the article we use the expression ``$(a,b,\dots)$-bounded" to mean that a quantity is bounded from above by a number depending only on the parameters $a,b,\dots$.
 A number of further results on commuting probability in finite groups can be found in \cite{gr,bgmn,DS-commprob,dlms}.

It is well-known that a finite group is nilpotent if and only if any two Sylow subgroups of coprime orders commute. The following theorem, established in \cite{dlms}, provides a probabilistic variation of this fact.

\begin{theorem}\label{old} 
If $\epsilon>0$ and $G$ is a finite group such that for any distinct primes  $p,q\in\pi(G)$ there is a Sylow $p$-subgroup $P$ and a Sylow $q$-subgroup $Q$ of $G$ for which $\pr(P,Q) \ge\epsilon$, then $F_2(G)$ has $\epsilon$-bounded index in $G$. 
\end{theorem}

As usual, here $F_i(G)$ stands for the $i$th term of the upper Fitting series of the group $G$.

In this paper we handle similar issues for finite groups admitting coprime automorphisms. An automorphism $\al$ of a finite group $G$ is said to be coprime if $(|G|,|\al|) = 1$. If a group $A$ acts on a group $G$, we write $[G,A]$ to denote the subgroup generated by all $g^{-1}g^\al$, where $\al\in A$ and $g\in G$. Note that $[G,A]$ is an $A$-invariant normal subgroup of $G$. If $A$ is a group of coprime automorphisms of $G$, then $[G,A]$ is nilpotent if and only if  $[P,A]$ and $[Q,A]$ commute whenever $P$ is an $A$-invariant Sylow $p$-subgroup and $Q$ is an $A$-invariant Sylow $q$-subgroup with $p\ne q$ (see for example \cite[Theorem 1.4]{blms2023}). Here we will prove
 
\begin{theorem}\label{mainA} Let $\epsilon>0$, and let $G$ be a finite group admitting a group of coprime automorphisms $A$ such that for any distinct primes  $p,q\in\pi(G)$ there is an $A$-invariant Sylow $p$-subgroup $P$ and an $A$-invariant Sylow $q$-subgroup $Q$ of $G$ for which $\pr([P,A],[Q,A])\ge\epsilon$. Then $F_2([G,A])$ has $\epsilon$-bounded index in $[G,A]$.
\end{theorem}

It is noteworthy that under the hypothesis of Theorem \ref{mainA} the index of $F([G,A])$  in $[G,A]$ can be arbitrarily large (see the example in Section 4).

Our next result is related to the theorem established in \cite{dgms} that if $P$ is a Sylow subgroup of a finite group $G$ such that $\Pr(P,P^x) \geq \epsilon$ for all $x \in G$, then the index $[P:O_p(G)]$ is $\epsilon$-bounded while $O_p(G)$ is bounded-by-abelian-by-bounded. We will prove the following theorem.

\begin{theorem} \label{Op} Let $G$ be a finite group admitting a group of coprime automorphisms $A$. Let $P$ be an $A$-invariant Sylow $p$-subgroup of $G$ and assume that \[\Pr([P,A], [P,A]^x) \geq \epsilon\] for all $x \in G$. Then the order of $[P,A]$ modulo $O_p(G)$ is $\epsilon$-bounded. 
\end{theorem} 

Next, we deal with groups $G$ in which the above condition holds for every prime divisor of the order of $G$. It turns out that in this case the structure of $G$ is as restricted as in P. M. Neumann's theorem.

\begin{theorem}\label{p_fitt} Let $G$ be a finite group admitting a group of coprime automorphisms $A$. Assume that $G=[G,A]$ and for any prime $p$ dividing the order of $G$ there is an $A$-invariant Sylow $p$-subgroups $P$ such that \[\Pr([P,A], [P,A]^x) \geq \epsilon\] for all $x \in G$. Then $G$ is bounded-by-abelian-by-bounded.
\end{theorem}
Unsurprisingly, the proofs of all main results in this paper depend on the classification of finite simple groups.

\section{Coprime action} 

We say that a group $A$ acts coprimely on a group $G$ if the automorphisms of $G$ induced by the elements of $A$ are coprime.
We denote by $C_G(\alpha)$ the fixed-point subgroup $\{x \in G \mid x^\alpha = x\}$ of an automorphism $\alpha$ and by $I_G(\alpha)$ the set of all elements of the form $g^{-1} g^\alpha$, where $g \in G$. Thus, $[G,\alpha]$ is generated by $I_G(\alpha)$. Observe that $|I_G(\alpha)| = |G:C_G(\alpha)|$.

In what follows we use the following well-known facts (see for example \cite{asch}), often without mention. 
\begin{lemma}\label{coprime}
Let a group $A$ act coprimely on a finite group $G$. The following  holds:
\begin{itemize}
    \item[(i)] $G = [G, A]C_G(A)$ and $[G,A]=[G,A,A]$;\
		\item[(ii)] if $G$ is abelian, then $G = [G,A]\times C_G(A)$;
    \item[(iii)] if $N$ is any $A$-invariant normal subgroup of $G$, we have $C_{G/N}(A) = C_G(A)N/N$;
		\item[(iv)] the group $G$ possesses an $A$-invariant Sylow $p$-subgroup for each prime $p \in \pi(G)$,  any two $A$-invariant Sylow $p$-subgroups are conjugate by an element of $C_G(A)$, and any  $A$-invariant  $p$-subgroup of $G$ is contained in an $A$-invariant Sylow $p$-subgroup;
		\item[(v)] if $N$ is any normal subgroup of $G$ such that $[N,A]=1$, then $[G,A]$ centralizes $N$.
    \end{itemize}
\end{lemma}

Throughout, by a simple group we mean a finite nonabelian simple group. We will often use without special references the well-known corollary of the classification that if a simple group $G$ admits a group of coprime automorphisms $A$ of order $e\ne1$, then $G = L(q)$ is a group of Lie type, $A=\langle\al\rangle$ is cyclic, and $\alpha$ is a field automorphism. Furthermore, $C_G(\alpha) = L(q_0)$ is a group of the same Lie type defined over a smaller field such that $q = q_0^e$ (see \cite{GLS3}).

\begin{lemma}\label{simple-order}
Let $C$ be a positive integer and $G$ a finite  simple group of Lie type in characteristic $p$ admitting a nontrivial coprime
automorphism $\al.$ Suppose that order of $[P, \alpha]$ is at most $C$ whenever $P$ is an $\alpha$-invariant Sylow $p$-subgroup of $G$. Then the order of $G$ is $C$-bounded.
\end{lemma}
\begin{proof} The proof follows the same lines as the proof of \cite[Lemma 2.4]{AGScoprime}. We will repeat it here for the reader's convenience.
Let $G = L(q)$ and let $C_G(\alpha) = L(q_0)$, where $q = q_0^e$ and $e=|\al|$. Note that $e\ge 3$, because $G$ is a finite simple group and $\alpha$ is coprime. We have that $|P|=q_0^{|{\al}|t}$ for some integer $t$ and $|C_P(\al)|=q_0^t$,  therefore  
\[|[P,\al]| \ge q_0^{t(|{\al}|-1)} > |P|^{1/2}.\]
 Comparing the orders of $G$ and $P$ (see e.g.    \cite[Table I]{GLS1}), we see that $|G| \le |P|^3$
   and the order of $G$ is $C$-bounded. 
\end{proof}
 We will need the following easy remark (see e.g. \cite{AGScoprime}).
\begin{lemma}\label{easy0} Assume that  $G$ is a finite group admitting a coprime automorphism $\alpha$ such that $G=[G,\alpha]$. If $|I_G(\alpha)| \leq m$, then the order of $G$ is $m$-bounded.\end{lemma}

In the sequel, we denote by $\pi(G)$ the set of prime divisors of the order of a group $G$ and,  somewhat abusing the terminology, 
 we assume that a Sylow $p$-subgroup is trivial if $p\not\in\pi(G)$. If $\pi$ is a set of primes, then $O_\pi(G)$ denotes the largest normal $\pi$-subgroup of $G$ and $O_{\pi'}(G)$ denotes the largest normal subgroup  of $G$ whose order is not divisible by primes in $\pi$. Moreover, $F(G)$ denotes the Fitting subgroup of $G$, which is the largest normal nilpotent subgroup of $G$.
Recall that, if $G$ is a finite soluble group, then $F(G)$ contains its centralizer. 
\begin{lemma}\label{easy1} Let  $G$ be a finite group admitting a coprime automorphism $\alpha$  such that $G=[G,\alpha]$. If $G$ is soluble, then $p\in\pi(G)$ if and only if $I_P(\al)\neq 1$ for some $\al$-invariant Sylow $p$-subgroup $P$ of $G$.
\end{lemma}
\begin{proof} Let $G$ be a counterexample of minimal order. Then $G$ possesses a nontrivial $\al$-invariant Sylow $p$-subgroup $P$ such that $I_P(\al)=1$. Because of minimality, $O_{p'}(G)=1$ and $C_G(O_p(G)) \le O_p(G)$. As $P\leq C_G(\al)$, it follows from Lemma \ref{coprime} that $O_p(G) \le Z(G)$. So $G \le C_G(O_p(G)) \le O_p(G)$ and therefore $G=P$. Since $G=[G, \al]$, we get that $G=1$, a contradiction.
\end{proof}

We remark that the above lemma fails for non-soluble groups. Indeed, let $q$ be an odd prime power such that the group $G=\PSL_2(q)$  admits a coprime automorphism $\alpha \neq 1$. Then  $C_G(\alpha) = \PSL_2(q_0)$, where $q = q_0^e$ and $e=|\al|$.  Note that $e$ is odd, because $\al$ is coprime,   therefore $|G|/|C_G(\al)| = (q_0^{2e}-1)/(q_0-1)$ is odd as well.  It follows that any $\alpha$-invariant Sylow $2$-subgroup of $G$ is contained in $C_G(\alpha)$. On the other hand, it is clear that $G=[G,\al]$.

We will need Theorem B of \cite{IsaacsHartley}. We record it in the next lemma for the reader's convenience.

If $A$ is a finite group and $k$ is a field of characteristic not dividing $|A|$, 
then any $k A$-module $V$ is a direct sum of simple components. Let  $S$ be a given simple $k A$-module and let $m_{A,S}(V)$  denote the number of simple components of $V$ 
isomorphic to $S$. 

\begin{lemma}\label{HIthmB} Let $A$ be an arbitrary finite group. Then there exists a number $\gamma = \gamma_A > 0$, depending only on $A$, with the following property: Let $A$ act coprimely on a finite soluble group $G$, and let $k$ be any field with characteristic not dividing $|A|$. Let $V$ be any simple $k AG$-module and let $S$ be any $k A$-module which appears as a component of the restriction $V_A$. Then $m_{A,S}(V) \geq \gamma \dim V$.
\end{lemma}

\begin{lemma}\label{easy2}   Let  $G$ be a finite soluble group admitting a coprime automorphism $\alpha$ such that $G=[G,\alpha]$.  Let $M$ be a minimal $\al$-invariant normal subgroup of $G$ and assume that $|I_M(\al)|\leq n$. Then either $M$ is central or the order of $M$ is $(|\al|,n)$-bounded.
\end{lemma}
\begin{proof}
 If $I_M(\al)= 1$, then $M\leq C_G(\al)$ and so  $M$ is contained in the centre  of $G$ by Lemma \ref{coprime}.

 Assume that $|I_M(\al)|\geq 2$. Because of the minimality of $M$ it follows that $M$ is an elementary abelian $p$-group for some prime $p$. Let $|M|=p^t$. We regard $M$ as an irreducible $G\langle\al\rangle$-module.  As $|I_M(\al)|\geq 2$, there is  a nontrivial simple 
 $\langle\al\rangle$-submodule $S$ of $M$. Let $\gamma$ be the constant as in Lemma \ref{HIthmB}, which only depends on $|\al|$. Note that $M=[M,\al]\times C_M(\al)$, by Lemma \ref{coprime}, and all $\langle\al\rangle$-submodules of $M$ isomorphic to $S$ are contained in $[M,\al]$. If  $S$ appears in $M$ with multiplicity $m_S$ then  ${\mathrm{dim}}[M,\al]\ge m_S\ge \gamma\, {\mathrm{dim}} M=\gamma\, t$.
Note that $|[M,\al]|=|M:C_M(\al)|=|I_M(\al)|\leq n$.
Therefore $n\ge |[M,\al]|\ge p^{\gamma t}$, which implies that $|M|=p^t\le n^{\frac 1\gamma}$ is $(|\al|,n)$-bounded.
\end{proof} 

A group is said metanilpotent if it possesses a normal nilpotent subgroup $N$ such that $G/N$ is nilpotent.
In the sequel $\gamma_\infty(G)$ stands for the intersection of the terms of the lower central series of a group $G$.

\begin{lemma}\label{gamma_infty}\cite[Lemma 2.4]{AST} If G is a finite metanilpotent group, then 
$\gamma_\infty(G) = \prod_p[K_p, H_{p'}]$, where $K_p$  is a Sylow $p$-subgroup of $\gamma_\infty(G)$ and $H_{p'}$ is a Hall $p'$-subgroup of $G$.
\end{lemma}

\begin{lemma}\label{easy3}  Assume that  $G$ is a finite soluble group admitting a coprime automorphism $\alpha$ such that $G=[G,\alpha]$. Let $D=\gamma_\infty(G)$ and suppose that $|I_D(\al)|\leq n$. Then the order of $D$ is $(|\al|,n)$-bounded.
\end{lemma}
\begin{proof} We assume that $D\neq 1.$ As $D=[D,G]$, it follows that $D$ is not central in $G$ and therefore $|I_D(\al)|\geq 2$ by Lemma \ref{coprime}.

We will argue by induction on $n$. Let $N$ be the maximal normal subgroup of $G$ contained in $C_G(\al)$. Note that $N\leq Z(G)$ by Lemma \ref{coprime}. We pass to the quotient $\bar G=G/N$ and let $\bar D=DN/N$. Let $\bar M$ be a minimal $\al$-invariant normal subgroup of $\bar G$ contained in $\bar D$. Observe that  $[\bar M,\alpha]\ne 1$, whence $I_{\bar D/\bar M}(\al) < I_D(\al)$ and, by induction, the index of $\bar M$ in $\bar D$ is $(|\al|,n)$-bounded. In view of Lemma \ref{easy2} the order of $\bar M$ is also $(|\al|,n)$-bounded. Hence $\bar D$ has $(|\al|,n)$-bounded order.
 
Going back to $G$ we get that the centre of $D$ has $(|\al|,n)$-bounded index in $D$. 
By Schur's theorem  \cite[4.12]{robinson2} the derived subgroup $D'$ has $(|\al|,n)$-bounded order. Passing to the quotient group $G/D'$, we may assume that $D$ is abelian and hence $G$ is metanilpotent. If $P$ is a Sylow $p$-subgroup of $D$, by Lemma \ref{gamma_infty} we have $P=[P,H]$, where $H$ is a Hall $p'$-subgroup of $G$.  As $P$ is abelian, $P=[P,H]\times C_P(H)$, whence $C_P(H)=1$ and so $P\cap Z(G)=1$. This holds for every prime divisor $p$ of the order of $D$. Therefore $D\cap Z(G)=1$. As $N\le Z(G)$ and the order of $D$ is $(|\al|,n)$-bounded modulo $N$, the lemma follows.
\end{proof} 

If $G$ is a finite soluble group, the Fitting height $h(G)$ of $G$ is the length of a shortest normal series all of whose quotients are nilpotent. In their seminal paper \cite{hh} Hall and Higman showed that a finite group of exponent $e$ possesses a normal series of $e$-bounded length all of whose quotients are either nilpotent or isomorphic to a direct product of nonabelian simple  groups. Therefore a finite soluble group of exponent $e$ has $e$-bounded Fitting height.

\begin{remark}\label{remark_FN} 
Let $N$ be a normal subgroup of $G$ and let  $K/N=F(G/N)$. If $|N| \le m$, then the index of $F(G)$ in $K$ is $m$-bounded. 
 \medskip
  
 {\rm  Indeed,  $C=C_K(N)$ is nilpotent and the index of $C$ in $K$ is at most $(m-1)!$.  It is clear that $C \le F(G)$ so the claim follows. }
\end{remark}

\begin{lemma}\label{fitt} Let $G$ be a finite soluble group admitting a coprime automorphism $\alpha$ such that $G=[G,\alpha]$. Assume that $|I_{F(G)}(\al)|\leq n$. Then $|G|$ is $n$-bounded.
\end{lemma}
\begin{proof}  Set $F=F(G)$. If $n=1$, then $F\leq C_G(\al)$ and so $F\leq Z(G)$, 
 whence $G=F\leq C_G(\al)$. Since $G=[G,\alpha]$, it follows that $G=1$ and the lemma holds. Therefore we assume that $n\geq2$ .

As $|I_F(\al)|\leq n$ and $\langle\al\rangle$ acts on $I_F(\al)$ by permuting its elements, the kernel of this action has index at most $n!$. Therefore there exists a positive integer  $j\leq n!$, such that $\al^j$ centralizes $I_{F}(\al)$.
Then $\al^j$ centralizes the whole subgroup $[F,\al]=\langle I_F(\al)\rangle$. 
 As $F=[F,\al]C_F(\al)$, it follows that $\al^j$ centralizes $F$. Therefore,  by Lemma \ref{coprime}, $[G,\alpha^j]$ centralizes $F$.  Since $G$ is soluble, we deduce that $[G,\alpha^j]\leq F$. 
Thus, 
 \[  [G, \al^j]=[G, \al^j, \al^j] \le [F,\al^j]=1. \] 
It follows that the automorphism $\al$ has order dividing $j$, which is $n$-bounded.  

Lemma \ref{easy2} implies that there is a normal series 
\[1=M_1\leq\dots\leq M_s=F,\]
 all of whose factors are either central in $G$ or of $(|\al|,n)$-bounded order. 
 Actually,  the non-central factors have $n$-bounded order because the order of $\al$ is $n$-bounded.  As $G/C_G(M_{i+1}/M_i)$ acts on each factor $M_{i+1}/M_i$ by automorphisms, there is an $n$-bounded number $e$ such that $G^e$ centralizes all factors $M_{i+1}/M_i$. It follows from Kaluzhnin's theorem \cite[Theorem 16.3.1]{KarMer} that $G^e/C_{G^e}(F)$ is nilpotent. Therefore $G^e$ is metanilpotent, as $C_{G^e}(F)\le F$. By the Hall-Higman theory \cite{hh} $G$ has $n$-bounded Fitting height $h=h(G)$.
 
 We now argue by induction on $h$. If $h=1$ then $G=F$ and the result follows from Lemma \ref{easy0}. So assume $h>1$.  Set $D=\gamma_\infty(G)$ and observe that $h(D)=h-1$. 
 Moreover $F([D,\al])$ is subnormal in $G$, hence contained in $F$. Thus, by induction, the order of $[D,\al]$ is $n$-bounded.
   We deduce from Lemma \ref{easy3} that the order of $D$ is $n$-bounded.

As $G/D$ is nilpotent, it follows from Remark \ref{remark_FN} that   the index of $F$ in $G$ is $n$-bounded. 
We have 
\[ |I_G(\al)|=|G:C_G(\al)|\le |G:F|\, |F:C_F(\al)|= |G:F|\,|I_F(\al)|.\]
 So $|I_G(\al)|$ is $n$-bounded and 
 Lemma \ref{easy0} yields the desired result. 
\end{proof}

\section{Commuting Probability}

If $X,Y$ are subsets of a finite group $G$, we have
\[ \pr(X,Y) =\frac{ | \{ (x,y) \in X\times Y \mid xy=yx \} |}{|X|\,|Y|}.\]

Note that 
 $ \pr(X,Y) = \pr(Y,X)$ and 
\[   \pr(X,Y) =\frac 1 { |Y|} \sum_{y\in Y}\frac{|C_{X}(y)| }{|X| }= \frac 1 {|X|} \sum_{x\in X} \frac{|C_{Y}(x)| }{|Y| },\]
where, as usual, $C_{Y}(x)$ denotes the set of all elements of $Y$ commuting with $x$.

The next lemma is essentially Lemma 2.2 of \cite{dlms} and it is useful when considering quotients, subgroups, or direct products of groups. In the sequel it is often used without explicit mention. 

\begin{lemma}\label{lem:basic} 
 Let $G$ be a finite group and let $H, K$ be subgroups of $G$.  Then 
\begin{enumerate}
\item If $N$ is a normal subgroup of $G$, then $\pr (HN/N,KN/N) \ge \pr (H,K)$.
\item If $H_0 \le H$, then $\pr (H_0,K) \ge \pr (H,K)\ge \frac 1{|H:H_0|}\pr (H_0,K)$. 
\item If $G=G_1 \times G_2$, $H_i \le G_i$   and $K_i \le G_i$, then 
\[\pr (H_1 \times H_2, K_1 \times K_2) = \pr (H_1, K_1) \pr(H_2, K_2).\] 
\end{enumerate}
\end{lemma}

The next lemma is Lemma 2.4 of \cite{dgms}.
\begin{lemma}\label{DLMS_pq}
Let $P$ be a  $p$-subgroup and $Q$ be a  $q$-subgroup of a finite group $G$. 
   If $[P,Q] \neq 1$, then  $ \pr (P,Q) \le 3/4.$
	\end{lemma}
	
We will also need some technical results,  which are analogous to some results in \cite{dlms,DS-commprob} for $A$-invariant subgroups, where $A$ is a group of automorphisms of $G$. 

\begin{lemma}\label{bound_m}
 Let $\ep>0$. There exists an $\epsilon$-bounded integer $m$ with the following property: If $G$ is a finite group with a group $A$ acting on $G$ by  automorphisms and  $H, K$ are $A$-invariant subgroups of $G$ with $\pr (H,K) \ge \epsilon >0 $, then
 there exists an $A$-invariant normal subgroup $H_0$ of $H$ such that:
 \begin{enumerate}
\item $|H:H_0| \le m$; 
\item 
  $|K: C_{K}(x)| \le m$  for every $x \in H_0$.
\end{enumerate}
\end{lemma}

\begin{proof} 
Note that the set 
 \[X=\{x\in H \mid  |x^{K}|\leq 2/\epsilon\}\]
is $A$-invariant. Following  line by line  the proof of Lemma 2.8 of \cite{dlms}, we find 
 an $A$-invariant subgroup $T\le H$ such that the indices $|H:T|$ and $|K: C_{K}(x)|$ are $\ep$-bounded for every $x \in T$. 
    
  Let $U$ be the maximal normal subgroup of $H$ contained in $T$. Clearly, the index $|H:U|$ is $\epsilon$-bounded. Since  each $U^{a}$ is normal in $H$, for every $a\in A$, there exist $\ep$-boundedly many elements $a_i\in A$ such that $U^A=\prod_i U^{a_i}$. Set $H_0=U^A$ and notice that $|K: C_{K}(x^{a_i})|$ is $\ep$-bounded  for every $x \in U$ and for every $a_i$. As the number of $a_i$ is $\ep$-bounded, we deduce that  $|K: C_{K}(x)|$ is $\ep$-bounded  for every $x \in H_0$. The result follows.
 \end{proof}

\begin{proposition}\label{mainAut} 
Let a group $A$ act on a finite group $G$, and let $H$ be an $A$-invariant subgroup of $G$ such that $\pr(H,G)\geq\epsilon>0$. Then there is an $A$-invariant normal subgroup $U\leq G$ and an $A$-invariant normal subgroup $B$ of $H$ such that the indices $[G:U]$, $[H:B]$, and the order of the commutator subgroup $[B,U]^G$ 
are $\epsilon$-bounded. 
\end{proposition}
\begin{proof} The proof uses the same arguments as those in the proof of Proposition 1.2 of \cite{DS-commprob}. 
We outline it here for the reader's convenience, pointing out which relevant subgroups are $A$-invariant, and refer to the paper \cite{DS-commprob} for further details. 

By Lemma \ref{bound_m} there exists an $A$-invariant normal subgroup $B$ of $H$  of $\ep$-bounded index in $H$ such that $|G: C_{G}(x)|$ is $\ep$-bounded for every $x \in B$. Set $L=\langle B^G\rangle$ and note that $L$ is  $A$-invariant. It follows from \cite[Theorem 1.1]{cri}  that the commutator subgroup $[L,L]$ has $\ep$-bounded order. By Lemma \ref{lem:basic} (i) we can replace $G$ with the quotient group $G/[L,L]$ and  still assume that $\pr(H,G)\geq\ep$
and the indices $|H:B|$ and $|G: C_{G}(x)|$  are $\ep$-bounded, for every $x \in B$. Therefore now $L$ is abelian.   

Again by Lemma \ref{bound_m}, as  $\pr(G,H)=\pr(H,G)\geq\epsilon$, 
there exists an $A$-invariant normal subgroup $U$ of $G$ such that  
 $|G:U|\le m$  and $|H: C_{H}(x)| \le m$  for every $x \in E$, where the integer $m$ is $\ep$-bounded.
Now $|b^G|$ is $\ep$-bounded, for every $b \in B$ and $|y^B| \le m$ for every $y\in U$. As $L$ is abelian, it follows from  Lemma 2.2 of \cite{DS-commprob} that $[B,U]$ has $\epsilon$-bounded order.
 
 Moreover $[B,U]$ is normalized by $U$, 
therefore it
 has $\ep$-boundedly many conjugates in $G$, all of them normalizing  each other. Hence, $[B,U]^G$ has $\ep$-bounded order. This concludes the proof. 
\end{proof}

\section{The soluble case} 

In the sequel, we will work with groups  satisfying the following hypothesis. 

\begin{hypothesis}\label{00} Let $\ep>0$ and let $G$ be a finite group admitting a  coprime automorphism $\al$ such that for any distinct primes $p,q\in\pi(G)$ there exists a  Sylow $p$-subgroup $P$ and a Sylow $q$-subgroup $Q$ in $G$, both $\al$-invariant,  for which $\pr([P,\al],[Q,\al])\geq\ep$. 
\end{hypothesis}

We note that under Hypothesis \ref{00} the index $|[G,\al]:F([G,\al])|$ can be arbitrarily large. Indeed, let $C$ be the cyclic group of order 3 and let $\al$ be the involutory automorphism of $C$. Let $p_1,p_2,\dots,p_s$ be distinct primes greater than $3$. For $p\in\{p_1,p_2,\dots,p_s\}$ let $C_p$ be the cyclic group of order $p$ and $B_p$ the base of the wreath product $C_p\wr C\langle\al\rangle$ of $C_p$ by $C\langle\al\rangle$. Let  $H_p=[B_pC,\al]$. Observe that $C<H_p$ and $\al$ induces an involutory automorphism of $H_p$ such that $H_p=[H_p,\al]$. Moreover, 
 $|H_p:F(H_p)|=3$. Now let $G$ be the direct product of $H_{p_i}$ for $i=1,\dots,s$. In a natural way $\al$ induces an involutory automorphism of $G$ such that $G=[G,\al]$. For any primes $p,q\in\pi(G)$ other than 3 the Sylow $p$-subgroup and Sylow $q$-subgroup of $G$ commute.  If $P$ is a Sylow $p$-subgroup for $p\geq5$ and $S$ an $\al$-invariant Sylow $3$-subgroup, then $\pr(P,S)>1/3$. Note that $|G:F(G)|=3^s$, which can be arbitrarily large.

Observe that if a group $G$ satisfies Hypothesis \ref{00} and $H$ is an $\al$-invariant normal subgroup of $G$, then $H$ satisfies Hypothesis \ref{00} as well.

\begin{remark}\label{remark_conj} 
 If $G$ satisfies  Hypothesis \ref{00} and $P$ is an $\al$-invariant Sylow $p$-subgroup of $G$, then for every $q \neq p$ there exists an $\al$-invariant  
  Sylow $q$-subgroup $Q$ of $G$ for which $\pr([P,\al],[Q,\al])\geq\ep$.  \medskip
 
{\rm  
 Indeed, any  two  $\al$-invariant Sylow $p$-subgroups of $G$  are conjugate by an element of $C_G(\al)$. So, if  $P^{x}$ and $Q$ are respectively an $\al$-invariant Sylow $p$-subgroup  and an $\al$-invariant  Sylow $q$-subgroup of $G$ such that $\pr([P^x,\al],[Q,\al])\geq\ep$ with    $x\in C_G(\al)$, 
  then 
  \[ \pr([P,\al],[Q^{x^{-1}},\al])=\pr([P^{x},\al],[Q,\al]) \geq\ep.\] 
	}
\end{remark}
 
\begin{lemma}\label{pq} Assume Hypothesis \ref{00} with $G=PQ$, where $P$ is a normal Sylow $p$-subgroup and $Q$ an $\al$-invariant Sylow $q$-subgroup such that  $Q=[Q,\al]$. Then there exists an $\ep$-bounded integer $m$  such that $|G:F(G)|\le m$.
\end{lemma}
\begin{proof}
As $F(G/P'O_q(G))=F(G)/P'O_q(G)$, we may assume that $F(G)=P$ and $P$ is abelian.

Observe that $|Q: C_Q(y)|=|G:C_G(y)| $ for every $y \in [P,\al]$.  So, 
    taking into account that $Q=[Q, \al]$ and that $\pr([P, \al], Q) \ge \ep$ by Remark \ref{remark_conj}, we have
\begin{eqnarray*} \pr([P, \al], G) &=& \frac{1}{|[P,\al]|} \sum_{y\in  [P, \al] } \frac{|C_G(y)|}{|G|} \\
 &=& \frac{1}{|[P,\al]|} \sum_{y\in  [P, \al] } \frac{|C_Q(y)|}{|Q|} 
 =  \pr([P, \al], Q) \ge \ep.
\end{eqnarray*}

By Proposition \ref{mainAut} there is an $A$-invariant normal subgroup $U\leq G$ and an $A$-invariant subgroup $P_0\leq [P,\al]$ such that the indices $|G:U|$ and $[[P,\al]:P_0]$, and the order of  $[P_0,U]^G$ are $\epsilon$-bounded.

Set $N=[P_0,U]^G$ and $K/N=F(G/N)$.

By Remark \ref{remark_FN}, the index of $F(G)$ in $K$ is $\epsilon$-bounded and,   in order to bound the index of $F(G)$ in $G$, it is enough to bound the index of $K$ in $G$. So we now assume that $N=[P_0,U]^G=1$. 

As $G=PQ$ and $P$ is abelian, $U$ acts coprimely on $[P,U]$.
 In particular, $C_{[P,U]} (U) =1$ and so
 $[P,U] \cap P_0=1$. Since $P_0$ has $\ep$-bounded index in $[P,\al]$, it follows that $[[P,U],\al]$ has $\ep$-bounded order. 
In particular $|I_{[P,U]}(\al)|$ is $\ep$-bounded.

Let $H=[P,U]Q$ and let $V=C_Q([P,U])$; then $V=O_q(H)$ and $F(H)=[P,U]V$.
 Let $\bar Q=Q/V$, $\bar H=H/V$ and note that $\bar H$ is isomorphic to the semidirect product $[P,U]\bar Q$. 
As $C_{\bar Q}([P,U])=1$, it follows that $F(\bar H)=[P,U]$. 
 
 Since $[P,U]=[P,U,U]\le [P,U,Q]=[P,U,\bar Q]$, we have that $\bar H=\bar Q^{\bar H}$. Observe that $\bar Q =[\bar Q, \al] \le [\bar H, \al]$. The latter subgroup is normal in $\bar H$ so it follows that $\bar H=[\bar H, \al ]$.  Moreover, $I_{F(\bar H)}(\al) =I_{[P,U]}(\al)$ and so $|I_{F(\bar H)}(\al)|$ is $\ep$-bounded. We are in a position to apply  Lemma \ref{fitt} and deduce that the order of $\bar H$ is $\ep$-bounded. So the order of $\bar Q$ is $\ep$-bounded as well. Therefore $V$ has $\ep$-bounded index in $Q$ and so $PV$ has $\ep$-bounded index in $G$. Moreover, $PV$ is a normal subgroup of $G$. 

We claim that  $R=PU \cap PV$ is nilpotent. Indeed, as $V$ acts coprimely on $P$, 
\[ [P, PU \cap PV]=[P, PU \cap PV, PU \cap PV] \le [[P,PU], V]=1. \] 
 Therefore $F(G)$ contains $R$ and thus has $\ep$-bounded index in $G$, as required.
\end{proof} 
In what follows we need the next observation.
\begin{remark}\label{metanilp}
Let $G$ be a finite metanilpotent group and $Q$ a Sylow $q$-subgroup of $G$. If $M=O_{q'}(F(G))$, then $C_Q(M) \le F(G)$. 
\end{remark}

\begin{lemma}\label{big} Under Hypothesis \ref{00} assume that $G$ is soluble and let $m$ be as in Lemma \ref{pq}. If  $q>m$ is a  prime and $Q$ is an $\al$-invariant Sylow $q$-subgroup, then $[Q,\al]\leq F(G)$.
\end{lemma}
\begin{proof}  Assume that the lemma is false and let $G$ be a counterexample of minimal order.  Let  $q >m$ be a prime and  let $Q$ be an $\al$-invariant Sylow $q$-subgroup of $G$ such that $[Q,\al]\not\le F(G)$. By considering the quotient group $G/F(G)$ and taking into account minimality, we get  that $[Q,\al]\le F_2(G)$. So, again by minimality, $G=F_2(G)$, that is, $G$ is metanilpotent.
By Remark \ref{remark_conj}, for every prime $p \neq q$ that divides $|F(G)|$, there exists   an $\al$-invariant Sylow $p$-subgroup $P$ of $G$ such that 
 $\pr( [P, \al],[Q,\al]) \ge \ep$.
 The  Sylow $p$-subgroup $P_1$ of $F(G)$ is contained in $P$ and so $\pr( [P_1, \al],[Q,\al]) \ge \ep$.
 Now we apply Lemma \ref{pq} to the subgroup $H=P_1[Q, \al]$ and deduce that $[Q, \al] \le F(H)$, as $q > m$. 
Since $P_1\le F(H)$, it follows that $[Q,\al]$ commutes with $P_1$. As this happens for every prime $p\ne q$,
 we conclude that  $[Q,\al]\leq F(G)$ (see  Remark \ref{metanilp}). 
 \end{proof}

We can now prove Theorem \ref{mainA} in the particular case where $G$ is soluble and $A$ is cyclic.

\begin{lemma}\label{soluble} Under Hypothesis \ref{00} assume that $G$ is soluble  and  $G=[G,\alpha]$.  
 Then the index $|G:F_2(G)|$ is $\ep$-bounded.
\end{lemma}
\begin{proof} 
Since $G$ is soluble and $G=[G,\al]$, it follows  from Lemma \ref{easy1}  that if $p$ divides the order of $G$, then there exists an $\al$-invariant  Sylow $p$-subgroup $P$ of $G$ such that $[P, \al] \neq 1.$ 

Let $Q$ be an $\al$-invariant Sylow $q$-subgroup of $G$ for a prime $q>m$, where $m$ is as in Lemma \ref{pq}. By Lemma \ref{big}, $[Q,\al]\leq F(G)$. Therefore, by Lemma \ref{easy1}, $q$ does not divide the order of $G/F(G)$.  

Passing to the quotient over $F(G)$, we are reduced to the case where the prime divisors of $|G|$ do not exceed $m$. Therefore $\pi(G)$ contains only $\ep$-boundedly many primes. We will show that under this assumption $F(G)$ has $\ep$-bounded index in $G$ and this will imply the desired result. 
  
 Let $N=F_2(G)$ and note that $N$ satisfies Hypothesis \ref{00}. 

  Let $Q$ be  an $\al$-invariant Sylow $q$-subgroup of $N$ and set 
  \[M= O_{q'}(F(G))=P_1 \times \cdots \times P_r,\] 
  where $P_i$ is a Sylow $p_i$-subgroup of $F(G)$. We know that $r$ is $\ep$-bounded. Note that each $P_i$ is contained in every $\al$-invariant Sylow $p_i$-subgroup of $G$, so by virtue of Lemma \ref{lem:basic} and Remark \ref{remark_conj}, the group $P_i [Q,\al]$ satisfies the hypotheses of  Lemma \ref{pq}. We deduce that the index of the centralizer $C_i$ of $P_i$ in $[Q,\al]$ is $\ep$-bounded.
 Thus the intersection $C=\cap C_i$ of all centralizers has $\ep$-bounded index in $[Q,\al]$ and it centralizes $M$. As $N$ is metanilpotent, $C \le F(N)=F(G)$  (see Remark \ref{metanilp}). 
 Therefore  the order of $[Q,\al] F(G)/F(G)$ is $\ep$-bounded. 

 This holds for every prime $q$ that divides the order of  $N/F(G)$, and there are  $\ep$-boundedly many such primes. As $N/F(G)$ is nilpotent, $[N,\al]$ is the direct product of the subgroups $[Q,\al] F(G)/F(G)$, hence $[N,\al]$  has $\ep$-bounded order modulo $F(G)$. This means that $[F(G/F(G)),\al]$ has $\ep$-bounded order. Now Lemma \ref{fitt} tells us that $G/F(G)$ has $\ep$-bounded order. This completes the proof.
\end{proof}

\section{The case when $G$ is a simple group} 

This section is devoted to the proof of Theorem \ref{mainA} when $G$ is a simple group 
 and $A=\langle \al\rangle\ne 1$ is cyclic. 
  Recall that in this case $G$ 
   is a group of Lie type defined over the field $\mathbb{F}_q$ and $\alpha$ is a  field  automorphism. Furthermore, $C_G(\alpha)$ 
    is a group of the same Lie type defined over a smaller field $\mathbb{F}_{q_0}$ such that $q = q_0^e$, where 
$e=|\al|$. 

We will use notation and terminology introduced 
in~\cite{carter1}, which we briefly recall hereafter.
 Let  $L(q)$
   be a finite simple Chevalley group, with  set of  roots $\Phi$ and  set of fundamental roots $\Pi=\{r_1,\dots,r_\ell\}$. For every root $r\in\Phi$ we denote by
\[
X_r = \{\, x_r(t) \mid t \in \mathbb{F}_q \,\}
\]
the corresponding root subgroup of $L(q)$. 

Any automorphism $\varphi$ of the field $\mathbb{F}_q$ induces a field 
automorphism (also denoted by $\varphi$) of $L(q)$ 
 defined by
\[
\big(x_r(t)\big)^\varphi \;=\; x_r\big(t^\varphi\big).
\]

We fix an ordering $\{r_1<r_2<\dots<r_\ell<\dots \}$ of the set $\Phi^+$ of positive roots. Then the subgroup
\[
U \;=\; \prod_{r\in\Phi^+} X_r
\]
is an $\alpha$-invariant Sylow $p$-subgroup of $L(q)$ 
 and every element $x\in U$ can be written in a unique way as a product
\[
x\;=\; \prod_{r\in\Phi^+} x_r(t_r), 
\]
with $t_r \in \mathbb{F}_q$. 

We will make use of the following observation. 

\begin{remark}\label{comm}  Let $x,y\in U$ and write
\[
x \;=\; x_{r_1}(t_1)\,\cdots\,x_{r_\ell}(t_\ell) \, z, 
\;\;\;
y \;=\; x_{r_1}(u_1)\,\cdots\,x_{r_\ell}(u_\ell) \, w,
\]
where $z,w\in\prod_{r\in (\Phi^+\setminus\Pi)} X_r$ and  $t_i,u_i\in\F_q$. 

Then, using  the  Chevalley commutator formulas \cite[Theorem 5.2.2]{carter1}
 we can write
\[xy \;=\; x_{r_1}(t_1+u_1)\,\cdots\,x_{r_\ell}(t_\ell+u_\ell) \, u, \quad \textrm{ with }\,\,u\in\prod_{r\in (\Phi^+\setminus\Pi)} X_r.
\]
\end{remark}

We will now recall some basic facts about twisted groups of Lie type.

Let $L(q^s)$ 
be a group of Lie type whose Dynkin diagram has a nontrivial symmetry 
$\rho$ of order $s$.  
If $\tau$ denotes the corresponding graph automorphism, suppose that 
$L(q^s)$ admits a nontrivial field automorphism $\varphi$ 
such that the  automorphism 
\[
\sigma = \varphi  \tau 
\]
satisfies $\sigma^s = 1$.  
Then, the twisted group
\[
^sL(q)
\]
is defined as the subgroup of $L(q^s)$ consisting of the elements fixed 
element-wise by $\sigma$.

The structure of $^sL(q)$ is very similar to that of a Chevalley group.  
If $\Phi$ is the root system of $L(q^s)$, the automorphism $\sigma$ 
determines a partition
\[
\Phi = \bigcup_i S_i.
\]
If $S$ is one of the equivalence classes in this partition, we define
\[
X_S = \langle\, X_r \mid r \in S \,\rangle \;\subseteq L(q^s),
\]
and 
\[
X_S^1 = \{\, x \in X_S \;\mid\; x^\sigma = x \,\} \;\subseteq {}^sL(q).
\]
The group ${}^sL(q)$ is generated by the subgroups $X_S^1$. In fact, the subgroups $X_S^1$ play the role of the 
root subgroups. In particular,
the subgroup
\[
U^1 \;=\; \prod_{S_i\subseteq\Phi^+} X_{S_i}^1
\]
is an $\alpha$-invariant Sylow $p$-subgroup of ${}^sL(q)$.

\begin{lemma}\label{regular}
Let $G$ be a group of Lie type in characteristic $p$ admitting a nontrivial coprime automorphism $\al$. Let $P$ be an $\al$-invariant  Sylow $p$-subgroup of $G$. 
Then $[P,\al]$ contains a regular unipotent element $x$ with $C_G(x)\le P$. 
\end{lemma}

\begin{proof}
First assume that $G=L(q)$ is an untwisted finite simple group of Lie type defined over the field $\mathbb{F}_q$. We may assume that the subgroup $P$ is the subgroup $U=\prod_{r\in \Phi^+}X_r$, defined above. Let $t\in \mathbb{F}_q$ be an element which is not fixed by $\alpha$, that is, $t^\alpha-t\ne 0$.
Consider the element 
\[x=\prod_{r\in \Pi}x_r(t^\alpha-t).\] 

 Then  $x$ is a regular unipotent element by Proposition 5.1.3 of \cite{carter2}. Note that $x\in [P,\al]$ because $x_r(t^\alpha-t)=x_{r}(t)^{-1}x_{r}(t)^\al\in  [P,\al]$ for each $r$.

Now assume that $G={}^sL(q)$ is of twisted type. We may assume that the subgroup $P$ is the subgroup $U^1=\prod_{S_i\subseteq\Phi^+} X_{S_i}^1$, defined above. 

Fix an equivalence class $S\subseteq\Phi^+$. 
By Proposition 13.6.3 of \cite{carter1}, we have that $S=\{a_1, \dots, a_j, a_{j+1}, \dots ,a_{k}\},$ with $j \le 3$, 
$\{a_1, \dots , a_j\} \subseteq \Pi$  and $\{ a_{j+1}, \dots ,a_{k}\} \in (\Phi^+\setminus\Pi).$ 
 Moreover there exist  automorphisms  $\varphi_i$ of $\mathbb{F}_{q^s}$ such that, for every  $t \in \mathbb{F}_{q^s}$, in  $X_S^1$ there exists an element of the form 
\[ x_S(t) = x_{a_1}(t) \cdots x_{a_j}(t^{\varphi_j}) \, z,  \quad \textrm{ with } z\in\prod_{r\in(\Phi^+\setminus\Pi)} X_r .\]  
 Let $t\in \mathbb{F}_q$ be an element which is not fixed by $\alpha$, that is, $t^\alpha-t\ne 0$. For each equivalence class $S_i\subseteq \Phi^+$  consider the element 
\[ x_i(t) = x_{a_1}(t) \cdots x_{a_j}(t^{\varphi_j}) \, z_i\in X_{S_i}^1, \quad \textrm{ with }  z_i\in\prod_{r\in(\Phi^+\setminus\Pi)} X_r\]  
  as above. 
 By Remark \ref{comm} we can write
\[ [x_i(t), \al] = x_{a_1}(t^\alpha-t)   \cdots  x_{a_j}(t^{\varphi_j \alpha}-t^{\varphi_j}) \, u_i \] 
 where  $t^{\varphi_i \alpha}-t^{\varphi_i} \neq 0$ and $u_i\in\prod_{r\in(\Phi^+\setminus\Pi)} X_r$. 
 Let 
 \[x=\prod_{S_i\subseteq\Phi^+}[x_i(t),\al].\]
  Again by Remark \ref{comm}, we can write $x$ in a unique way as
 \[ x = \prod_{r\in \Pi} x_r(t_r^\alpha-t_r)\,  u ,\] 
with $t_r^\alpha-t_r\ne 0$ and
 $u\in\prod_{r\in(\Phi^+\setminus\Pi)} X_r$. 
 Then  $x$ is a regular unipotent element by Proposition 5.1.3 of \cite{carter2}. Note that $x\in [P,\al]$ because $[x_i(t), \al]\in  [P,\al]$ for each $i$.

 The fact that, in both the twisted and untwisted cases,  $C_G(x) \le P$ follows from Corollary 4.6 of \cite{Hum}. 
\end{proof}

The following corollary is a straightforward consequence of the previous lemma.
\begin{corollary}\label{centraliz}
Let $G$ be a group of Lie type in characteristic $p$ admitting a nontrivial coprime automorphism $\al$. Let $P$ be an $\al$-invariant  Sylow $p$-subgroup of $G$. 
Then $C_G([P,\al]) \le P $.
\end{corollary}

If $n$ is a natural number and $p$ is a prime, the $p$-part $n_p$ (respectively, the $p'$-part $n_{p'}$) of $n$ is the largest $p$-power dividing $n$  (respectively, the largest divisor of $n$ coprime to $p$). 

\begin{lemma}\label{simple-p} 
  Under Hypothesis \ref{00} with $\al \neq 1$,
   assume that $G$ is simple. Then $G$ is a group of Lie type and
 the characteristic $p$ of $G$ is  $\ep$-bounded.
\end{lemma}
\begin{proof}
Recall that $G=L(q)$ is a finite simple group  of Lie type and $q=q_0^e=p^{t e}$, where $e=|\al| \ge 3$, because the order of a coprime automorphism of a simple group must be odd. 

Let $m$ be as in Lemma \ref{bound_m} and assume that $p>m$. Let $U$ be  an $\alpha$-invariant Sylow $p$-subgroup of $G$. 
Let  $r \neq p$ be a prime dividing $|G:\Gal|$, and let  $R$ be an $\alpha$-invariant Sylow $r$-subgroup of $G$
 such that $\pr([U,\al],[R,\al])\geq\ep$. 
  Since $p >m$,  $[U,\al]$ centralizes a normal subgroup $R_0$ of index at most $m$ in $R$.  
   It follows from Corollary \ref{centraliz}  that $R_0 \le U$, whence  $R_0=1$ and $|[R,\al]| \le m$. 
    Since $R=[R,\al]C_R(\al)$, it follows that for every prime divisor $r\neq p$ of $|G:\Gal|$ we have that the $r$-part of $|G:\Gal|$ is at most $m$. This implies that  $r \le m$. 
Thus  the $p'$-part of $|G:\Gal|$ has order at most $m^m$.  
  Checking the values of $|L(q_0^e):L(q_0)|$ (see \cite[TABLE I, p. 8]{GLS1}),
  we see that the $p'$-part $ |G:\Gal|_{p'} $ of $|G:\Gal|$ 
  is at least $q_0$. 
  Therefore $q_0 \le |G:\Gal|_{p'} \le  m^m$, which proves that $p$ is $\ep$-bounded. 
 \end{proof}

\begin{proposition}\label{simple} Under Hypothesis \ref{00} with $\al \neq 1$, 
 assume that $G$ is simple. 
 Then $|G|$ is $\ep$-bounded.
\end{proposition}
\begin{proof} As $G$ is simple,  it is a group of Lie type defined over a field of size $q$ and $\Gal$ is a finite simple group of the same Lie type and same rank, defined over a field of size $q_0 $ such that $q=q_0^{|\al|}=   p^{f |\al|}$, for some prime $p$ and some integer $f$. 
   By Lemma  \ref{simple-p} we know that  $p$ is $\ep$-bounded. 
  
According to \cite{gur-shar-wood} for every finite simple group $G$  in the list of \cite[Table 2]{gur-shar-wood}, 
 there is  an exponent $e$ such that a Zsigmondy prime $r$ for $(q, e)$, that is a primitive prime divisor of $q^e-1$, divides the order of $G$ but does not divide the order of any parabolic subgroup of $G$. 
    
    Assume that $G$ is one of the groups  in \cite[Table 2]{gur-shar-wood}, and let $e$ be the corresponding exponent. 
  We remark that, as 
   $q^e=p^{ f |{\al}| e}$,  a Zsigmondy prime for $(p, f |{\al}| e)$ is also a Zsigmondy prime  for $(q,e)$ and therefore divides the order of $G$.

    Consider  a  Zsigmondy prime $r$  for $(p, f |{\al}| e)$ and note that $r\ge f |{\al}| e$, as the order of $p$ modulo $r$ is precisely  $f |{\al}| e$.
  As $r$ does not divide $p^t-1$ for $t < f |{\al}| e$, in particular $r$ does not divide $q_0^s -1$ for $s <|{\al}| e$, and therefore $r$ does not divide the order of $\Gal$ (see e.g. \cite[Table I]{GLS1}).
   It follows that, if $R$ is an $\al$-invariant  $r$-subgroup of $G$, then $[R, \al]=R$. 
    
     Moreover,  since the centralizer of a unipotent element is always contained in a parabolic subgroup (see e.g.   \cite[Theorem 26.5, (Borel-Tits)]{malle-test}) we deduce that there is no $r$-element  that  centralizes a $p$-element of $G$. 
   Therefore, if $P$ is an $\al$-invariant Sylow $p$-subgroup of $G$  such that $\pr([P,\al],R)=\pr([P,\al],[R,\al])\geq\ep$, then 
   \[ \{ (x,y) \in [P, \al]\times R \mid [x,y]=1\}= 
   \{ (x,1 ) \mid x \in [P, \al] \} \cup  \{ (1,y ) \mid y \in R \} \]
    and 
\[ \pr([P, \al], R) = \frac1{|R|} + \frac 1{|[P,\al]|}- \frac 1{|[P,\al]| |R|} \ge \ep.\]
 Thus  one of 
  $|R|$ and $|[P,\alpha]|$ is $\ep$-bounded.
  
  If $|R|$, and hence $r$, is $\ep$-bounded,  then $f |{\al}| e$ is $\ep$-bounded, as  $r\ge f |{\al}| e$.
   So the rank of $G$ is $\ep$-bounded, since $e$ is defined as in \cite[Table 2]{gur-shar-wood}. Moreover, as $q=p^{ f |{\al}|}$, also $q$ is $\ep$-bounded. Hence, we conclude that the order of $G$ is $\ep$-bounded. 
   
   If $|[P,\al] |$ is $\ep$-bounded, then the order of $G$ is $\ep$-bounded by Lemma \ref{simple-order}. 

We are left with the cases where $G$ is not in the list of \cite[Table 2]{gur-shar-wood}. 
 As discussed in \cite[Section 4.1]{gur-shar-wood}, we have that $G = L_2(p)$ for
some Mersenne prime $p$ or the type of $G$ is one of 
 \[G_2(2), A^+_5 (2), A^-_2 (2),
A^-_3 (2), B_3(2), C_3(2), D^+_4 (2).\]
  In all these cases, $q=q_0=p$, against the fact that $\al$ is a nontrivial field automorphism of $G$.  
  The proof is now complete.
\end{proof}

\section{The case when $A$ is cyclic} 
This section is devoted to the proof of Theorem \ref{mainA} when $A$ is cyclic.

\begin{lemma}\label{general_pq} Assume that $G$ is  group admitting a coprime automorphism $\alpha$ such that $G=[G,\al]$ and for any distinct primes $p,q\in\pi(G)$ there exists a  Sylow $p$-subgroup $P$ and a Sylow $q$-subgroup $Q$ in $G$, both $\al$-invariant,  for which $[\, [P,\al],[Q,\al] \,]=1$. Then $G$ is nilpotent.
\end{lemma}
 
\begin{proof} Assume by contradiction that that the statement is false and let $G$ be a counterexample of minimal order. Let $N_1$, $N_2$ be two minimal normal minimal $\alpha$-invariant subgroups of $G$, then $G/N_1$ and $G/N_2$ are both nilpotent. If $N_1\ne N_2$, then $G$ embeds into the direct product $G/N_1\times G/N_2$, which is nilpotent, a contradiction. So $G$ has a unique minimal normal  $\alpha$-invariant subgroup $N$, which is  either elementary abelian or a direct product of isomorphic (nonabelian) simple groups. 

Suppose that $N$ is an elementary abelian $q$-group. Let $Q$ be a Sylow $q$-subgroup of $G$; as $G/N$ is nilpotent and $[G/N,\al]=G/N$, it follows that $Q$ is normal in $G$. Moreover $N\le Z(Q)$, by minimality. 
 
Let $P$ be a Sylow $p$-subgroup of $G$ with $p\ne q$. As $[G/N,\al]=G/N$ and $G/N$ is nilpotent, it follows that $[P,\al]N=PN.$ Therefore $[P,\al]=P$ for every Sylow $p$-subgroup $P$ with $p\ne q$. 

Since $G/N$ is nilpotent, $N$ is not contained in $Z(G)$. So, by minimality, $N \cap Z(G)=1$. By assumptions, for any prime $p\ne q$ there exists an $\al$-invariant  Sylow $p$-subgroup $P$ of $G$, for which $[P,[N,\al]]=1$.  Moreover $[N, \al] \le N \le Z(Q)$. It follows that $ [N, \al] \le Z(G)$. Then $[N, \al] \le N \cap Z(G)=1$ and, by Lemma \ref{coprime}, we conclude that $N \le Z(G)$, a contradiction.
 
So $N$ is a direct product of simple groups and $G=N$, by minimality. Again by minimality, $G$  must be a simple group. 

Therefore $G$ is a simple group of Lie type, say in characteristic $p$, and $\alpha$ is a field automorphism. Let $P$ be an $\alpha$-invariant Sylow $p$-subgroup of $G$. Note that $[P, \alpha]\neq 1$. Recall that by Corollary \ref{centraliz} the centralizer of $[P, \alpha]$ is contained in $P$. Therefore for every prime $q\ne p$ there exists an $\alpha$-invariant Sylow $q$-subgroup $Q$ of $G$ such that $[Q,\alpha] =1$ and $Q=\Qal \le  C_G(\alpha)$. Thus the index $|G:C_G(\alpha)|$ is a $p$-power, a contradiction. 
\end{proof}

The following corollary is a straightforward consequence of the above lemma. 
\begin{corollary}\label{simple_pq} Assume that $G$ is a simple group admitting a nontrivial coprime automorphism $\alpha$. Then there exist two distinct primes $p,q\in\pi(G)$ such that for every $\al$-invariant Sylow $p$-subgroup $P$ and every
$\al$-invariant Sylow $q$-subgroup $Q$ we have that $[P,\al]$ does not commute with $[Q,\al]$.
\end{corollary}

By a semisimple group we mean the direct product of finite simple groups. 

\begin{lemma}\label{semisimple-base} 
Under Hypothesis \ref{00},  assume that $G$ is semisimple and has no nontrivial proper $\al$-invariant normal subgroups.  
 Then the order of $G$ is $\ep$-bounded.
\end{lemma}
\begin{proof} 
 Note that the assumptions imply that $G=[G,\al]$. If $G$ is simple, the lemma follows from Proposition \ref{simple}. So we assume that 
 \[G=S\times S^\al\dots\times S^{\al^{s-1}},\]
  where $S$ is a simple factor and $s\geq 2$. 

Note that $\al^{s}$ normalizes $S$ and if $P$ is an $\al^{s}$-invariant Sylow $p$-subgroup of $S$, then $P^{\langle\al\rangle}$ is an  $\al$-invariant Sylow $p$-subgroup  of $G$. As any two $\al$-invariant  Sylow $p$-subgroups  of $G$ are conjugate by an element of $C_G(\al)$, it follows that any $\al$-invariant  Sylow $p$-subgroup  of $G$ is of the form $P^{\langle\al\rangle}$, where $P$ is an $\al^s$-invariant  Sylow $p$-subgroup  of $S$. 
  
So let $P^{\langle\al\rangle}, Q^{\langle\al\rangle}$ be as in Hypothesis \ref{00}. Note that if $x\in S$ and $\bar\pi: G\to S$ is the projection on $S$, then $\bar\pi([x,\al])=x^{-1}$. Therefore $\bar\pi([P^{\langle\al\rangle},\al])=P$ and $\bar\pi([Q^{\langle\al\rangle},\al])=Q$. 
It follows from Lemma \ref{lem:basic} that
\[\pr(P,Q)\ge \pr([P^{\langle\al\rangle},\al],[Q^{\langle\al\rangle},\al])\ge \ep.\]
This proves that for any distinct primes  $p,q\in\pi(S)$ there is a Sylow $p$-subgroup $P$ and a Sylow $q$-subgroup $Q$ of $S$ such that $\pr(P,Q) \ge \epsilon$.
It follows from Theorem 1.1  of \cite{dlms} that the order of $S$ is $\ep$-bounded. 

Since $P^{\langle\al\rangle}=P \times P^\al \times\cdots \times P^{\al^{s-1}}$, it follows that $C_{  P^{\langle\al\rangle}}(\al)$ consists of elements of the form $(x,x^\al, \dots,x^{\al^{s-1}})$, where $x$ ranges over $C_P(\al^s)$. Since $\al $ is coprime, $P^{\langle\al\rangle}=[P^{\langle\al\rangle}, \al] C_{  P^{\langle\al\rangle}}(\al)$, and so 
 $[P^{\langle\al\rangle}, \al] $ has index at most $|P|$ in $P^{\langle\al\rangle}$. 
 Similarly, $|Q^{\langle\al\rangle}:[Q^{\langle\al\rangle},\al]| \le |Q|$. 
 Therefore, by Lemma \ref{lem:basic} (2), we obtain that
\[ \pr(P^{\langle\al\rangle}, Q^{\langle\al\rangle}) 
\ge\frac{1}{|P|} \frac{1}{|Q|} \pr \left( [P^{\langle\al\rangle},\al],[Q^{\langle\al\rangle},\al] \right)\geq\frac{1}{|S|}\ep.\] 
 Since  $|S|$ is $\ep$-bounded, we deduce that there exists $\tilde \ep$, depending only on $\ep$, such that for any distinct primes  $p,q\in\pi(G)$  
 there is a Sylow $p$-subgroup $\tilde P$ and a Sylow $q$-subgroup $\tilde Q$ of $G$ such that $\pr(\tilde P, \tilde Q) \ge \tilde \epsilon$.
Now the result follows from Theorem 1.1  of \cite{dlms}. This completes the proof.
\end{proof}

\begin{lemma}\label{semisimple} Under Hypothesis \ref{00}, assume that $G$ is semisimple and $G=[G,\alpha]$.  
Then $|G|$ is $\ep$-bounded.
\end{lemma}
\begin{proof} 
Write 
\[ G=T_1\times\dots\times T_s, \] 
where $T_i$ are minimal normal $\al$-invariant subgroups of $G$. 
Note that $T_i=[T_i,\al]$. 
 As each $T_i$ is a  normal  subgroup of $G$,  Hypothesis \ref{00} holds for every $T_i$. 
It follows from   Lemma \ref{semisimple-base} that each $T_i$ has $\ep$-bounded order. Therefore we only need to prove that $s$ is $\ep$-bounded.

Note that only $\ep$-boundedly many groups occur as simple factors of $T_i$ and there are only $\ep$-boundedly many pairwise non-isomorphic groups among the $T_i$. So without loss of generality we can assume that all the $T_i$ are isomorphic to a given semisimple group $T$.

First assume that  $T$ is simple. Let $p$ and $q$ be as in  Corollary \ref{simple_pq}, so that 
 for any $\al$-invariant Sylow $p$-subgroup $\tilde P$ and any $\al$-invariant Sylow $q$-subgroup $\tilde Q$ of $T$ we have that $[\tilde P,\al]$ does not commute with $[\tilde Q,\al]$. It follows from Lemma \ref{DLMS_pq} that $\pr([\tilde P,\al], [\tilde Q,\al]) \le 3/4$.
 Choose $P$, $Q$ as in  Hypothesis \ref{00} and write $P=\prod P_i$, $Q=\prod Q_i$, where the $P_i$ (resp. $Q_i$) are  $\al$-invariant Sylow $p$-subgroups (resp. $q$-subgroups) of $T_i$.
Then, by Lemma \ref{lem:basic} (3),
\[\epsilon\le \pr( [P, \al],[Q, \al])=\prod_{i=1}^s \pr([P_i, \al],[Q_i, \al])\le (3/4)^s.\]
Therefore $s$ is $\ep$-bounded.

So assume that each $T_i$ is not simple. Let
\[ T=S\times S^\al \times \dots\times S^{\al^m} \] 
 where $S$ is a simple group and $m \ge 1$. 
  By Proposition 3.1 of \cite{dlms} there exist two distinct primes $p$ and $q$ such that for every Sylow $p$-subgroup $P_0$
 and every Sylow $q$-subgroup $Q_0$ of $S$, we have $\pr(P_0, Q_0) \le 2/3$. 
 Choose  $P$, $Q$ as in  Hypothesis \ref{00} with respect to these primes $p$ and $q$. 
 Let $P_i = P \cap T_i$ and $Q_i=Q \cap T_i$, so that 
 \[ P= \prod_{i=1}^s P_i \quad \textrm{ and } \quad  Q= \prod_{i=1}^s Q_i.\]  
For a given index $i\in\{1,\dots,s\},$ let $\bar\pi: T_i\to S$ be the projection on the first component of $T_i$ and and set 
 $P_0=P\cap S=\bar\pi(P_i)$ and $Q_0=Q\cap S=\bar\pi(Q_i).$ 
  If $x\in S$ then $\bar\pi([x,\al])=x^{-1}$.
 Therefore $\bar\pi([P_i,\al])=P_0$ and $\bar\pi([Q_i,\al])=Q_0$.
Thus
\[\pr([P_i, \al],[Q_i, \al])\le \pr(P_0, Q_0) \le 2/3.\]

This holds for every $T_i$ and by Lemma \ref{lem:basic} (3) we get 
 \[\epsilon\le \pr( [P, \al],[Q, \al])=\prod_{i=1}^s \pr([P_i, \al],[Q_i, \al])\le (2/3)^s.\]
Therefore $s$ is $\ep$-bounded and the proof is complete. 
\end{proof}

Recall that the generalized Fitting subgroup $F^*(G)$ of a finite group
$G$ is the product of the Fitting subgroup $F(G)$ and all subnormal quasisimple subgroups; here a group is quasisimple if it is perfect and its
quotient by the centre is a nonabelian simple group. 
In every finite group $G$, the centralizer of $F^*(G)$ is contained in $F^*(G)$. We will denote by $R(G)$ the soluble radical of the finite group $G$, i.e. the largest normal soluble subgroup of  $G$.

\begin{lemma}\label{noradical} Under Hypothesis \ref{00}, assume that the soluble radical of $G$ is trivial and  $G=[G,\alpha]$.  
Then $|G|$ is $\ep$-bounded.
\end{lemma}
\begin{proof}  Taking into account that $R(G)=1$ write $F^*(G)=S_1\times\dots\times S_l$, where the factors $S_i$ are simple.   
 The group $G\langle\al\rangle$ acts as a permutation group of the set of simple factors of $F^*(G)$. Let $K$ be the kernel of this action. Note that by Lemma \ref{semisimple} the order of $[F^*(G),\al]$ is $\ep$-bounded. Hence the element $\al$ moves only $\ep$-boundedly many points. It follows from Lemma 2.5 of \cite{AGScoprime} that $G/K$ has $\ep$-bounded order, so there are only $\ep$-boundedly many conjugates of $[F^*(G),\al]$, all of them normalizing each other, being normal in $F^*(G)$. Therefore $N=[F^*(G),\al]^G$ has $\ep$-bounded order. Since $\al$ centralizes $F^*(G)/N$, by Lemma \ref{coprime} $F^*(G)/N$ is central in $G/N$. Taking into account that $F^*(G)$ is semisimple we deduce that $F^*(G)=N$ and so $F^*(G)$ has $\ep$-bounded order. Since $G$ acts on $F^*(G)$ by conjugation and $F^*(G)$ contains its centralizer, the lemma follows.
\end{proof}

We can now prove Theorem \ref{mainA} in the special case when $A$ is cyclic.

\begin{theorem}\label{main} Under Hypothesis \ref{00}, assume that $G=[G,\alpha]$.  Then the index $[G:F_2(G)]$ is $\ep$-bounded.
\end{theorem}
\begin{proof} Let $R=R(G)$ be the soluble radical of $G$. By Lemma \ref{soluble} the index $|[R,\al]:F_2([R,\al])|$ is $\ep$-bounded. As $F_2([R,\al])$ is subnormal in $G$ and therefore contained in $F_2(G)$, we deduce that 
  the order of $[R,\al]$ modulo $F_2(G)$ is $\ep$-bounded. We know from Lemma \ref{noradical} that the index $|G:R|$ is $\ep$-bounded. As $R$ normalizes $[R,\al]$, there are $\ep$-boundedly many $G$-conjugates of $[R,\al]$ and they normalize each other, being all normal in $R$. Therefore the order of $[R,\al]^G$ modulo $F_2(G)$ is $\ep$-bounded.
Since  $\al$ centralizes $R/[R,\al]^G$, Lemma \ref{easy0} now implies that $R/F_2(G)$ has $\ep$-bounded order, and the result follows.
\end{proof}

\section{The general case}

To extend  Theorem \ref{main} from the case of one automorphism $\al$ to the case where a coprime group of automorphisms $A$ acts on $G$, we need the following lemma. 
 
\begin{lemma}\label{lemmaA} Let $A$ be a coprime group of automorphisms of a finite group $G$ and assume that $|[G,a]|\leq m$ for every $a\in A$. Then $|[G,A]|$ is $m$-bounded.
\end{lemma}
\begin{proof} 
 Note that, as $ [G,\al]$ is normal in $G$ for all $\al \in A$, 
\[ [G,A]=\prod_{\al \in A} [G,\al].\] 
 So, it is sufficient to show that $|A|$ is $m$-bounded. Since the order of $A$ is bounded in terms of the orders of its abelian subgroups (see
for instance \cite[Theorem 5.2]{AI}), we can assume that $A$ is abelian. Set $K=GA$. Then $|a^K|\leq m$ for every $a\in A$. 
 
 It follows from Theorem 1.1 of \cite{cri} that  the derived subgroup of $\langle A^G \rangle$ has  $m$-bounded order. 
 Then $[G, A]=[G,A,A] \le [A^G, A]$ has  $m$-bounded order as well. As $A$ acts on $[G,A]$, we deduce that the index of $C_A([G,A])$ in $A$ is $m$-bounded.
  Note that if $a \in   C_A([G,A])$, then $  [G,a,a]=[G,a]=1$, hence $a=1$. It follows that the order of $A$ is $m$-bounded.
\end{proof}

Now we are ready to prove Theorem \ref{mainA}: 
\begin{proof}[Proof of Theorem \ref{mainA}]
Let $\ep>0$ and let $G$ be a finite group admitting a coprime automorphism group  $A$ such that   $G=[G,A]$ and for any distinct primes $p,q\in\pi(G)$ there are $A$-invariant Sylow $p$-subgroup $P$ and Sylow $q$-subgroup $Q$ in $G$ for which $\pr([P,A],[Q,A])\geq\ep$. 
 We want to prove that  $[G:F_2(G)]$ is $\ep$-bounded. 
 
 Recall that $G=[G,A]=\prod_{\al \in A} [G,\al].$
 Now we will show that for any $\al \in A$, the subgroup  $ [G,\al]$ satisfies the assumptions of Theorem  \ref{main}. Let $p,q\in\pi(G)$ be distinct primes and let 
$P$ and $Q$ be $A$-invariant Sylow $p$ and $q$ subgroups such that  $\pr([P,A],[Q,A])\geq\ep$; then $\pr([P,\al],[Q,\al])\geq\ep$. 
Set $\tilde P=P\cap[G,\al]$ and $\tilde Q=Q\cap[G,\al]$. Observe that $\tilde P$ and $\tilde Q$ are $\al$-invariant Sylow subgroups of $[G,\al]$.
We have 
\[  \pr([\tilde P,\al],[\tilde Q,\al]) \geq \pr([P,\al],[Q,\al])\geq\ep.\]
By Theorem \ref{main} we deduce that the index of $F_2( [G,\al])$ in $ [G,\al]$ is $\ep$-bounded. Note that $F_2( [G,\al])$ is characteristic in $ [G,\al]$, which is normal in $G$, therefore $F_2( [G,\al])\le F_2(G)$.

Now we factor out $F_2(G)$ and we get that $ [G,\al]$ has $\ep$-bounded order for every $\al \in A$. 
It follows from Lemma \ref{lemmaA} that the order of 
$ [G, A]=G$ is $\ep$-bounded. This concludes the proof.
\end{proof}

\section{Theorems \ref{Op} and \ref{p_fitt}} 

This section is devoted to the proofs of Theorem \ref{Op} and Theorem \ref{p_fitt}. We start by studying the case when $G$ is a finite simple group.

\begin{proposition} \label{Op-simple} Let $G$ be a finite simple group of Lie type in characteristic $p$ admitting a coprime automorphism $\alpha$. Let $P$ be an $\alpha$-invariant Sylow $p$-subgroup of $G$ and assume that
 \[\Pr([P, \alpha], [P, \alpha]^x) \geq \epsilon\]
  for all $x \in G$. Then the order of $[P, \alpha]$  is $\epsilon$-bounded. 
\end{proposition}
\begin{proof} 
Let $G = L(q)$  and $C_G(\al)=L(q_0)$ where  $q = q_0^e$, with $e=|\al|$. 
We can assume $\al \neq 1$. 
Let $m$ be the $\ep$-bounded integer in Lemma \ref{bound_m}, which we may assume to be bigger than $8$.

 If the rank $r$ of $G$ is at most $m,$ then by virtue of \cite[Theorem 1.3]{dgms} the order of $[P, \alpha]$ is $\epsilon$-bounded.
 
 So assume $r>m$. Then  $G$ contains the alternating group ${\mathrm{Alt} }(m)$ (see the proof of Lemma 6.4 of \cite{AGKS}) and therefore the order of $G$ is divisible by all primes less or equal than $m$. Since $(|\al|, |G|)=1$, we deduce that every prime divisor of the order of $\al$ is bigger than $m$. 
 
We know that $\al$ normalizes a Borel subgroup $B = UT,$ where $T$ is a
torus and $U$ is a Sylow $p$-subgroup. Then $B^-=U^-T$ is the
opposite Borel subgroup (with $U\cap U^-)=1$). This is obtained by conjugating $B$ by
the longest element $n$ in the Weyl group. We can assume that $P=U$. Let $x\in C_G(\al)$ be an element corresponding to $n$.
Note that, in particular, $P\cap P^x=1$. 

By Lemma  \ref{bound_m},  $ [P,\al]^x$ contains an $\al$-invariant normal subgroup $P_0$ 
 such that $|[P,\al]^x/P_0| \le m$ and $|[P,\al]: C_{[P,\al]}(g)| \le m$ for every $g \in P_0$.
 Note that $[P,\al]^x=[P^x,\al]$.
Since the order of $\al$ is divisible only by primes bigger than $m$, it acts trivially on the quotient group $[P^x,\al]/P_0$.
Therefore $[P^x,\al]=[P^x,\al,\al]\le P_0$, and hence $P_0=[P^x,\al]$. 

It follows from Lemma \ref{regular} that $[P^x,\al]$ contains a regular element $y$ with  $C_{G}(y)\le P^x$, so 
 \[ C_{[P,\al]}(y)\le P\cap P^x=1. \] 
  Since  $|[P,\al]: C_{[P,\al]}(y)| \le m$, we deduce that 
  $[P,\al]$ has $\ep$-bounded order, as claimed.
\end{proof}

We are now ready to prove Theorem \ref{Op}, which we restate here for convenience
\begin{thmOp}  
Let $G$ be a finite group admitting a group of coprime automorphisms $A$. Let $P$ be an $A$-invariant Sylow $p$-subgroup of $G$ and assume that \[\Pr([P,A], [P,A]^x) \geq \epsilon\] for all $x \in G$. Then the order of $[P,A]$ modulo $O_p(G)$ is $\epsilon$-bounded. 
\end{thmOp}
\begin{proof}  
We may assume that $O_p(G)=1$. Let $r$ be the maximum integer such that $G$ has a composition factor which is isomorphic to a simple group of Lie type of rank $r$ in characteristic $p$. We will show that $r$ is $\epsilon$-bounded. Then the theorem will be immediate from  \cite[Theorem 1.3]{dgms}.  
 We can assume $r >8$. 
 
First, we deal with the case where $A=\langle\al\rangle$ is cyclic.
 
Let $M$ be the characteristic subgroup of $G$ generated by all normal subgroups containing no composition factor which is isomorphic to a simple group of Lie type of rank $r$ in characteristic $p$. Note that the maximum integer $s$ such that $G/M$ has a composition factor which is isomorphic to a simple group of Lie type of rank $s$ in characteristic $p$ is precisely $r$. So we can pass to the quotient $G/M$ and without loss of generality assume that $M=1$. 

Then $G$ has at least one minimal $\al$-invariant normal subgroup which is a direct product of simple groups of Lie type of rank $r$ in characteristic $p$. Let $N$ be any such subgroup. 
 Then $N$ contains an $\al$-invariant normal subgroup  $ N_0=S^{\langle \al \rangle}$, for some simple group $S$ of Lie type of rank $r$ in characteristic $p$. 
Note that $P\cap N_0 $ is an $\alpha$-invariant Sylow $p$-subgroup of $N_0$ with 
 \[\Pr([P\cap N_0, \alpha], [P\cap N_0, \alpha]^x) \geq \epsilon\] for all $x \in N_0$. 
 
 If $N_0=S$, then the rank $r$ is $\ep$-bounded by Proposition \ref{Op-simple} and Lemma \ref{simple-order}. 
 
Otherwise 
 \[ N_0=S\times S^\alpha\times\cdots\times S^{\alpha^{c-1}}\]
   and an $\alpha$-invariant Sylow $p$-subgroup $P\cap  N_0$ of $ N_0$ is of the form $\tilde P\times \tilde P^\alpha\times\cdots\times \tilde P^{\alpha^{c-1}}$, where $\tilde P$ is an $\al^c$-invariant Sylow $p$-subgroup of $S$. Note that, if $x\in \tilde P$, then $x^{-1}x^\al\in [P\cap  N_0, \alpha]$ has the same order as $x$. 
 Therefore the exponent of  $\tilde P$ is the same as the exponent of  $[P\cap  N_0, \alpha],$ which is $\ep$-bounded by  Lemma 2.11 of \cite{dgms}. Note that  every classical group of Lie rank $r>8$
     contains the alternating group ${\mathrm{Alt}}(r)$ of degree $r$ (see the proof of Lemma 6.4 of \cite{AGKS}); as the exponent of the Sylow $p$-subgroup of  ${\mathrm{Alt}}(r)$ goes to infinity as $r$ does, it follows that the rank $r$ of $S$ is bounded in terms of the exponent of  $\tilde P$. This implies that $r$ is $\ep$-bounded.
     
Thus, the theorem holds in the case where $A$ is cyclic. The general case is straightforward from Lemma \ref{lemmaA} applied to the action of $A$ on $P$. \end{proof}

\begin{lemma}\label{product} Assume that the finite group $G$ is the product of $k$ normal subgroups $A_1,\dots, A_k$ such that
$\pr(A_i,A_j)\ge\epsilon$ for all $i,j$. Then $G$ has a normal subgroup $D$ of $(\ep,k)$-bounded index such that the order of $D'$ is $(\ep,k)$-bounded.
\end{lemma}
\begin{proof}
For any pair $i,j$, by virtue of Lemma \ref{bound_m}, $G$ contains a normal subgroup $D_{i,j}\leq A_i$ such that the indices $|A_i:D_{i,j}|$ and $|A_j:C_{A_j}(x)|$ are $\ep$-bounded for every $x\in D_{i,j}$.
 Let 
 \[ D_i=\bigcap_{j=1}^k D_{i,j}. \]  
 Then $D_i$ is a normal subgroup of $G$ contained in $A_i$ such that its index in $A_i$ is $(\ep,k)$-bounded and every element $x$ of $D_i$ has $(\ep,k)$-boundedly many $A_j$-conjugates, for every $j=1,\dots,k$. 
   Moreover, as $G=A_1\cdots A_k$ and $D_i$ is normal in $G$, the size of the conjugacy class 
   \[ |x^G| = | ( \cdots (x^{A_1})  \cdots )^{A_k}| \]  
   is $(\ep,k)$-bounded for every $x \in D_i$. 
   
      Note that the normal subgroup $D=D_1\cdots D_k$ has $(\ep,k)$-bounded index in $G$. Moreover, every element of $D$ is a product of at most $k$ elements from $D_1, \dots , D_k$, so it   has $(\ep,k)$-boundedly many $G$-conjugates.
			
It follows follows from the proof of Neumann's BFC-theorem \cite{neu} that the order of $D'$ is $(\ep,k)$-bounded (see also \cite{GuMa,wie}).
\end{proof}

We can now prove our last main result, Theorem \ref{p_fitt}.
\begin{thmB} Let $G$ be a finite group admitting a group of coprime automorphisms $A$. Assume that $G=[G,A]$ and for any prime $p$ dividing the order of $G$ there is an $A$-invariant Sylow $p$-subgroups $P$ such that \[\Pr([P,A], [P,A]^x) \geq \epsilon\] for all $x \in G$. Then $G$ is bounded-by-abelian-by-bounded.
\end{thmB}
\begin{proof} Let $F=F(G)$. It follows from Proposition \ref{Op} that there exists some $\ep$-bounded integer  such that
\[|[P,\al]F/F|\le m\]
for every $\al$-invariant Sylow subgroup $P$ of $G$. Then the index $|G : F|$ is $\epsilon$-bounded by Lemma 2.4 of \cite{AGScoprime}.

 We will now show that $[F,\al]^G$
 has $\ep$-bounded index in $G$. Consider the quotient group $\bar G=G/[F,\al]^G$.
  As $\bar G$ is the union of $\ep$-boundedly many cosets $\bar F\bar x_1, \dots, \bar F\bar x_k$, it follows that $I_{\bar G}(\al)=\{[\bar x_1,\al],\dots, [\bar x_k,\al]\}$ has $\ep$-bounded cardinality, therefore $\bar G$ has $\ep$-bounded order by Lemma \ref{easy0}, as required. 

So now it is sufficient to prove that $N=[F,\al]^G$ is  bounded-by-abelian-by--bounded. 
As $[F,\al]$ is normal in $F$, which has $\ep$-bounded index in $G$, it follows that $N$ is a product of $\ep$-boundedly many normal subgroups, all conjugate to $[F,\al]$, 
 say 
 \[ N= \prod_{i=1}^t [F,\al]^{g_i}, \]
  where $g_i \in G$ and $t$ is $\ep$-bounded. 
If $p$ is a prime and $P$ is the Sylow $p$-subgroup of $F$, then $[P, \al]$ is the Sylow $p$-subgroup of $[F, \al]$. Actually $[P, \al]=O_p([F,\al])$ and so 
$[P, \al]$ is normal in $F$.
 Therefore the Sylow $p$-subgroup $[P, \al]^{G}$ of $N$ is the product of  $[P, \al]^{g_i}$, for $i=1, \dots , t$,  where each $[P, \al]^{g_i}$  is normal in $N$ and $t$ is $\ep$-bounded. Since $P$ is contained in a Sylow $p$-subgroup of $G$, by the assumptions we have that 
\[ \pr \left([P, \al]^{g_i},[P, \al]^{g_j}\right)\ge\epsilon \]
for every $i, j=1, \dots , t$.  It follows from Lemma \ref{product} that $[P, \al]^{G}$ is bounded-by-abelian-by-bounded. 

Since the bounds do not depend on the prime $p$, it follows that for $p$ large enough $[P, \al]^{G}$ 
 is abelian. 
So we deduce that $N$ is a direct product of $\ep$-boundedly many bounded-by-abelian-by-bounded Sylow subgroups and an abelian subgroup. 
 We conclude that $N$ is bounded-by-abelian-by-bounded. The proof is complete.
\end{proof}

\end{document}